\documentclass[11pt] {article}      
\usepackage{amsmath,amssymb, amsthm, eucal}  
\usepackage{latexsym, graphicx, amscd}

\usepackage{txfonts}

\usepackage[T1]{fontenc}     
\usepackage{enumerate}       
\usepackage{float}                 
\usepackage{tikz}
\usetikzlibrary{matrix, arrows, calc}

\theoremstyle{plain}
\newtheorem{theorem}{Theorem}[section]            
\newtheorem{proposition}[theorem]{Proposition}  

\newtheorem{corollary}[theorem]{Corollary}	      %

\theoremstyle{definition}
    
\newtheorem{remark}[theorem]{Remark}	
\newtheorem{example}{Example}[section]

\numberwithin{theorem}{section}
\numberwithin{equation}{section}
\numberwithin{figure}{section}
\newcommand{\gaction}[2]{\genfrac{}{}{0.5pt}{}{#1}{#2}%
                        \!\lower2pt\hbox{\rotatebox[origin=c]{-90}{{$\looparrowright$}}}}
\newcommand{\dotfraction}[2]{\genfrac{}{}{0.5pt}{}{#1}{#2}%
                        \!\lower.5pt\hbox{{$\circ$}}}

\usepackage{titlesec}                                                           
\titlelabel{\thetitle.\ }                                                         
\titleformat*{\section}{\fontsize{14pt}{14pt} \bf}        
\usepackage{fancyhdr}
\usepackage{sidecap}  
\usepackage[hang,small,sc]{caption}

\begin{document}

\title{On a Diophantine equation that generates all integral Apollonian Gaskets}

\author{Jerzy Kocik\\
{\small Department of Mathematics, Southern Illinois University, Carbondale, IL 62901}\\
{\small jkocik@siu.edu}} 

\date{}

\sloppy

\maketitle



\begin{abstract} 
A remarkably simple Diophantine quadratic equation is known to generate all Apollonian integral gaskets (disk packings).  A new derivation of this formula is presented here based on inversive geometry.  Also occurrences of Pythagorean triples in such gaskets is discussed.
\end{abstract}

\section{Introduction}  
\label{s:1}

Apollonian disk packing (or Apollonian gasket) is a pattern obtained by starting with three mutually tangent circles of which one contains the other two, then recursively inscribing new circles (disks) in the curvilinear triangular regions (called ``ideal triangles'') formed between the circles.  Figure 1 shows a few examples, including (a) a special case of the noncompact ``Apollonian Strip'', (b) the Apollonian Window which is the only case that has symmetry $D_2$, (c) the regular threefold gasket, which has symmetry $D_3$, and (d) a general gasket that may have no mirror symmetry.

\begin{figure}[htbp] 
  \centering
  \includegraphics[width=5.67in]{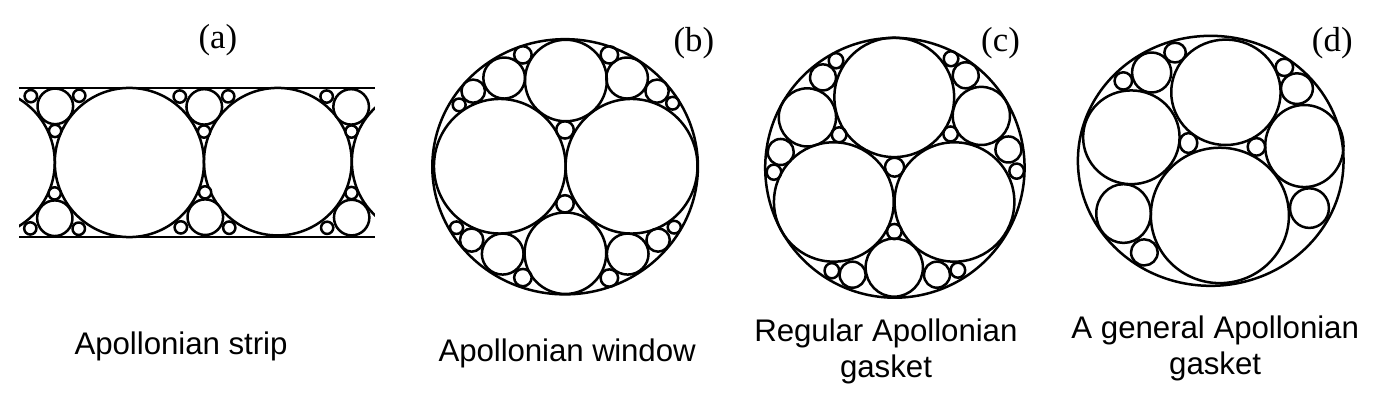}
  \caption{Various Apollonian disk packings}
  \label{fig:fig-1}
\end{figure}

It turns out that there exist an infinite number of \textit{integral} Apollonian packings, where the curvature of every circle/disk is an integer. 

Recall that four pairwise tangent circles are said to be in \textit{Descartes configuration}.  Suppose $a,b,c,d$ are curvatures of four such circles.  Then Descartes' circle formula states 
\begin{equation}    \label{eq:1.1}
  2(a^2+b^2+c^2+d^2) = (a+b+c+d)^2.
\end{equation}
One may use this formula to obtain the curvature of the fourth circle $d$, given the first three.  Being a quadratic equation, it yields two solutions, say $d_1$ and $d_2$, such that 
\begin{equation}    \label{eq:1.2}
  d_1+d_2 = 2(a+b+c).
\end{equation}
Disks $d_1$ and $d_2$ will be called \textit{Boyd-dual}.  As observed by Boyd \cite{B}, this ``linearized'' version of \eqref{eq:1.1} provides a handy tool for determining the curvatures in a particular Apollonian gasket starting with the first four values. In particular, it follows that if the first four values are integers, so are all. 

\begin{remark}
For consistency, the curvature of the circle that contains all of the remaining ones must be given a negative value for equations \eqref{eq:1.1} and \eqref{eq:1.2} to hold. It is customary to replace the term ``curvature'' by ``bend'' to account for this convention. Another way of looking at this is to think of disks rather than circles, where the greatest circle is the boundary of an exterior, unbounded, region.  This way no two disks in an Apollonian gasket overlap.
\end{remark}

In this note, Apollonian gaskets will be labeled by the bends of the five greatest circles, i.e., by the five least bends.  Why five will become clear when we consider symmetries, Appendix A. 

Examples of integral Apollonian gaskets include these:
\[
\begin{alignedat}{2}
  &(0,\ 0,\ 1,\ 1,\ 1) 
	&\quad &- \hbox{ Apollonian belt (Figure \ref{fig:fig-1}a)} \\
  &(-1,\ 2,\ 2,\ 3,\ 3) 
	&\quad &- \hbox{ Apollonian window (Figure \ref{fig:fig-1}c)} \\
  &(-2,\ 3,\ 6,\ 7,\ 7) 
	&\quad &- \hbox{ less regular gasket, but with $D_1$ symmetry} \\
  &(-6,\ 11,\ 14,\ 15,\ 23) 
	&\quad &- \hbox{ quite irregular case}
\end{alignedat}
\]
Note that the regular gasket (see Figure \ref{fig:fig-1}) cannot have integer bends, as its quintet of curvatures is, up to scale, $(1-\surd 3,  \ 2,\  2,\ 2,\ 10+2\surd 3)$, hence its curvatures are populated by elements of $\mathbb{Z}[\surd 3]$.  Due to \eqref{eq:1.2}, the integrality of the first four circles determines integrality of all disks in the packing.  An integral Apollonian packing is \textbf{irreducible} if the bends have no common factor except 1.  

The problem is to classify and determine all irreducible integral Apollonian gaskets.

\section{Integral disk packing -- the formula}  
\label{s:2}

All integral Apollonian disk packings may be determined using a simple Diophantine quadratic equation with constraints. The derivation of this formula is a much simpler alternative to that of ``super-Apollonian packing'' \cite{GLMWY1}--\cite{GLMWY3} and is based on inversive geometry.
\newpage
\begin{theorem}  \label{th:2.1}  
There is a one to one correspondence between the irreducible integral Apollonian gaskets and the irreducible quadruples of non-negative integers $B,k,n,\mu \in \mathbb{N}$ that are solutions to quadratic equation
\begin{equation}    \label{eq:2.1}
  B^2+\mu^2 = kn
\end{equation}
with constraints   
\[
\begin{alignedat}{2}
  \hbox{(i)} &&\quad &0 \le \mu \le B/\surd 3, \\
  \hbox{(ii)} &&\quad &2\mu \le k \le n.
\end{alignedat}
\]
Every solution to \eqref{eq:2.1} corresponds to an Apollonian gasket with the following quintet of the major bends (curvatures):
\[
  (B_0,\,B_1\,,B_2,\,B_3,\,B_4)
	= (-B,\, B+k,\, B+n,\, B+k+n-2\mu,\, B+k+n+2\mu)
\]

\begin{figure}[tbp] 
  \centering
  \includegraphics[width=3in,keepaspectratio]{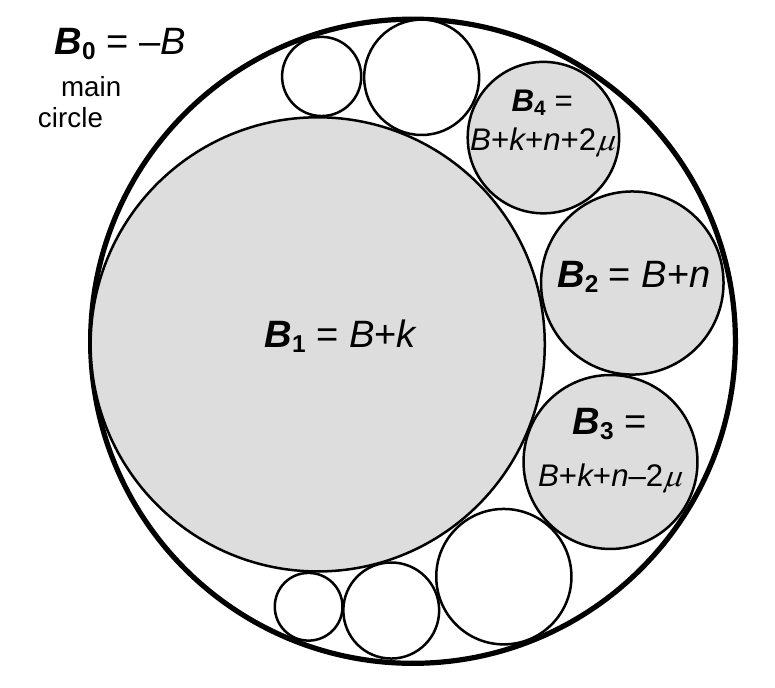}
  \caption{An Apollonian gasket and its four greatest circles (smallest curvatures)}
  \label{fig:fig-2}
\end{figure}

\end{theorem}

\begin{corollary}
The Apollonian gasket is integral iff gcd$(B,k,n) = 1$.
\end{corollary}

Figure \ref{fig:fig-2} locates the curvatures of the theorem in the Apollonian gasket.  Note that the triple of integers $(B$, $\mu$, $k)$ is also a good candidate for a label that uniquely identifies an Apollonian gasket (since $n$ is determined by: $n = (B^2 + \mu^2)/k)$.

Equation \eqref{eq:2.1} leads to an algorithm producing all integral Apollonian gaskets, ordered by the curvatures, presented in Figure \ref{fig:fig-3a}:

\begin{figure}[h!] 
  \centering
  \includegraphics[width=4.5in,keepaspectratio]{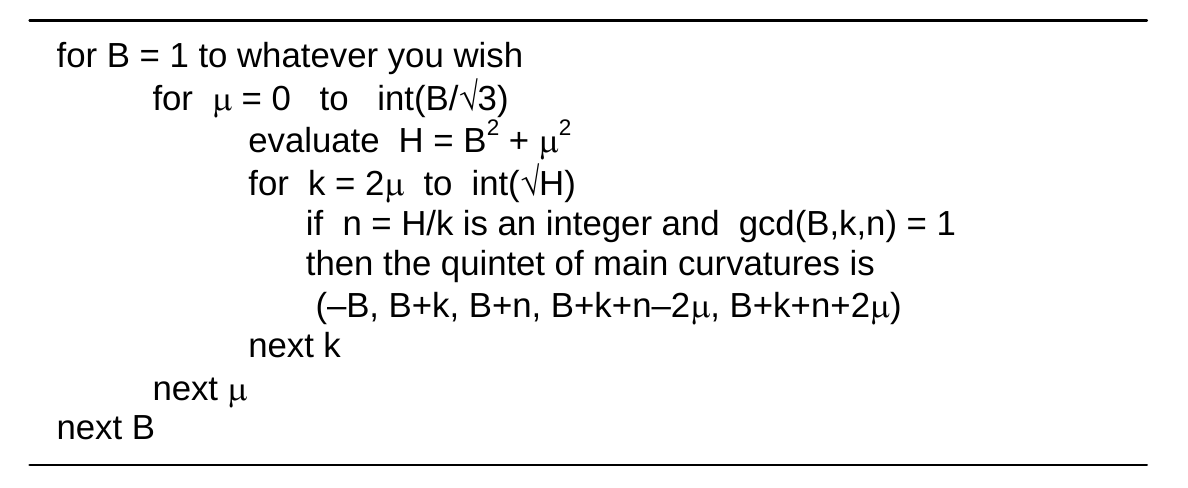}
  \caption{An algorithm that produces all integral disk packings.}
  \label{fig:fig-3a}
\end{figure}

Appendix A shows the first 183 entries for the main bend varying from 0 through 32.

\begin{example} 
The table below shows that there are only 3 Apollonian gaskets for $B = 6$.
\end{example}

\begin{figure}[H]
  \centering
  \includegraphics[width=5.67in,height=3.78in,keepaspectratio]{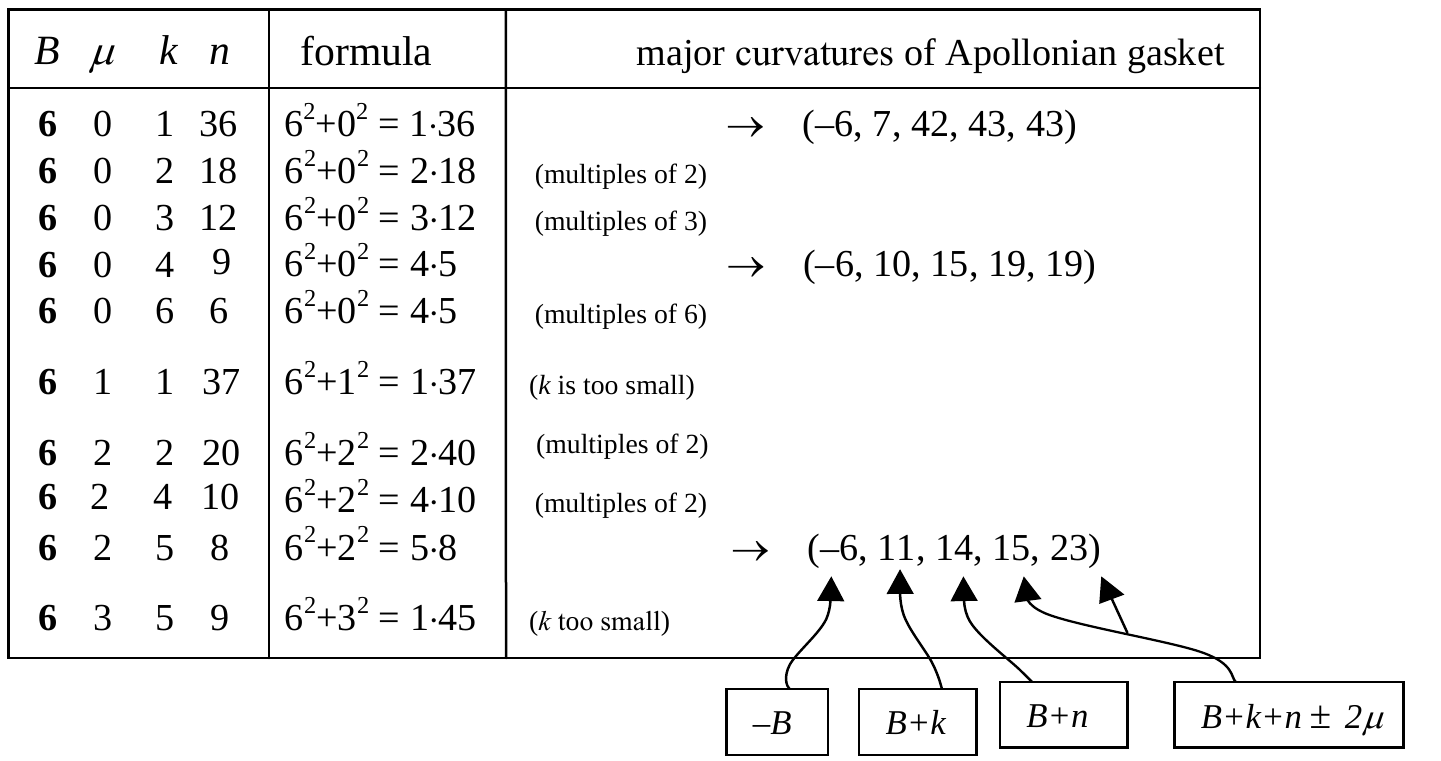}
  \label{fig:fig-3b}
\end{figure}

\begin{proof}[Proof Theorem 2.1] 
(We denote a circle and its curvature by the same symbol). Consider an Apollonian gasket of disks inscribed inside a circle of curvature $B$ (bend equal to $-B$). Draw an axis through the center of this circle and the center of the next largest circle $B_1$ (the horizontal axis $A$ in Figure \ref{fig:fig-4}).  Inverting the gasket through a circle $K$ of radius $2/B$ will produce an Apollonian belt, shown in the figure on the right side of the gasket. Denote its width by $2\rho$. Lines $L_0$ and $L_1$ are the images of $B$ and $B_1$, respectively.

\begin{figure}[htbp] 
  \centering
  \includegraphics[width=3.5in,keepaspectratio]{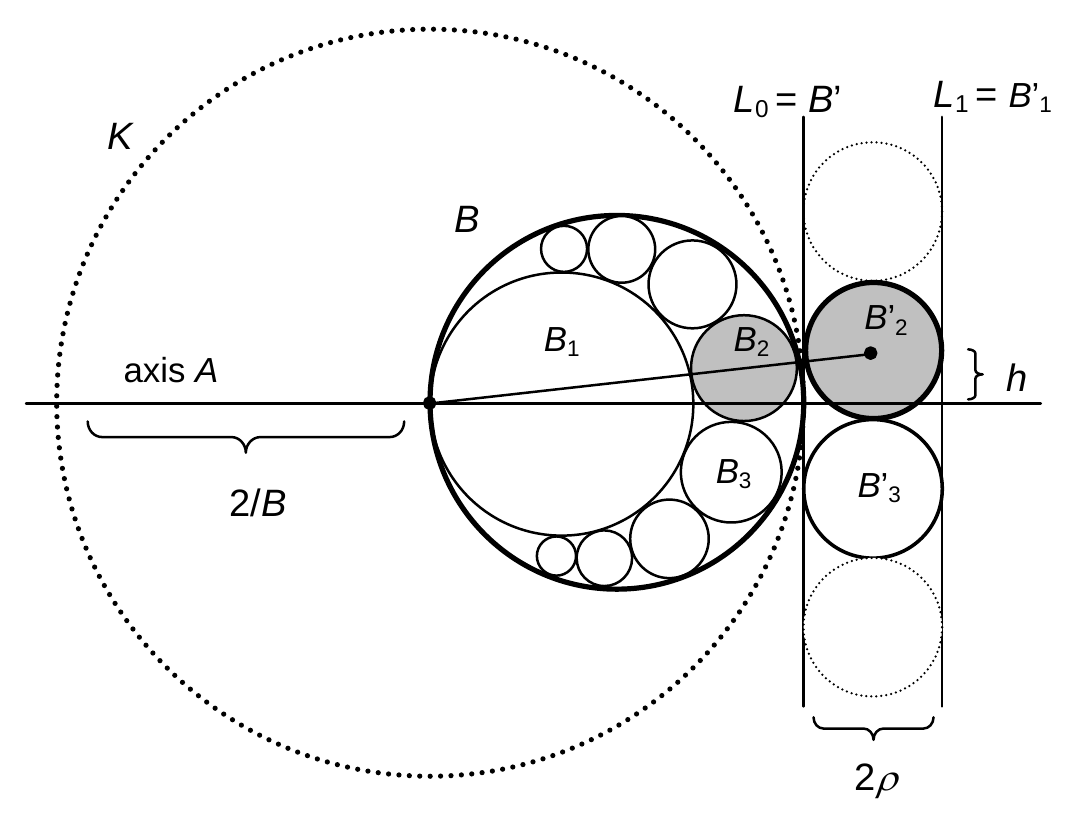}
  \caption{The method of inversion.}
  \label{fig:fig-4}
\end{figure}

The next largest circle, $B_2$, will show up in the strip as a circle $B'_2$\,.  It will intersect the axis $A$.  Various Apollonian gaskets will result by varying the strip's width (that is the parameter $\rho$) and the height $h$ of the center of the circle $B_2'$ above the axis.  Clearly, the range of $h$ is the interval $[0, \rho)$, as going beyond would result in repetitions of the Apollonian arrangements. 

Recall the basic formulas of inversive geometry: inversion through a circle of radius $K$ centered at $O$ carries a point at distance $d$ from $O$ to a point at distance $d'$ according to
\begin{equation}    \label{eq:2.2}
  dd' = K^2
\end{equation}
and a disk of radius $r$ with center distance $d$ from $O$ is mapped to a disk of radius $r'$ with its center located at distance $d'$ according to
\begin{equation}    \label{eq:2.3}
  \hbox{({\it i})} \ \  r' = \dfrac{rR^2}{d^2-r^2}\,, \qquad
  \hbox{({\it ii})} \ \  d' = \dfrac{dR^2}{d^2-r^2}\,.
\end{equation}
We shall analyze the Apollonian gasket as the inversive image of the Apollonian strip, disk by disk: 
\vskip7pt
\noindent
\textbf{Step 0.}  Recall that $B\in N$ represents the curvature of the encompassing circle of the gasket.
\vskip7pt
\noindent
\textbf{Step 1}.  To calculate the curvature $B_1$ of the greatest circle inside (the image of $L_1$), apply \eqref{eq:2.3}i) to the point of intersection of $L_1$ and axis $A$ to get:
\[
  B_1 = B + \underbrace{B^2 \rho}_{k} \,.
\]
Since $B$ is an integer, so is the last term; denote it by  $k = B^2 \rho$.  Clearly, $k$ can be any nonnegative integer (to make $B_1 \geq B$).
\vskip7pt
\noindent
\textbf{Step 2.}   For $B_2$, the image of $B'_2$, use \eqref{eq:2.3}ii to get 
\[
\begin{alignedat}{2}
  B_2 &= \dfrac{1}{r'} 
	= \dfrac{1}{\rho} \cdot \dfrac{d^2-\rho^2}{(2/B)^2} 
	= \dfrac{B^2}{4\rho} \cdot (d^2-\rho^2)
	&\qquad &\hbox{(simplification)} \\
   &= \dfrac{B^2}{4\rho} \cdot \big( (2/B+\rho)^2
	+ h^2 - \rho^2 \big)
	&\qquad &\hbox{(Pythagorean thm)} \\
   &= B + \underbrace{\dfrac{4+h^2B^2}{4\rho}}_{n}
\end{alignedat}
\]
As before, we conclude that the last term must be an integer; denote it by $n$. Clearly,  $n \geq k$ (to make $B_2 \geq B_1$).

\vskip7pt

\noindent
\textbf{Step 3.}  Similarly, we get a formula for the next largest circle $B_3$ located below $B_2$, the image of $B'_3$.  Simply use the above formula with $h' = h - 2\rho$ instead of $h$ to get 
\[
\begin{aligned}
  B_3 &= \dfrac{1}{\rho} \cdot \dfrac{d^2-\rho^2}{(2/B)^2} 
	= \dfrac{B^2}{4\rho} \cdot \big( (2/B+\rho)^2
	+ (h^2 - \rho^2)^2 - \rho^2 \big) \\
   &= B \ + \  \underbrace{\rho B^2\vphantom{\dfrac{h^2}{4\rho}}}_{k} 
	 \ +\  \underbrace{\dfrac{4+h^2B}{4\rho}}_{n}
	 \ - \  \underbrace{hB^2\vphantom{\dfrac{h^2}{4\rho}}}_{m}
\end{aligned}
\]
Quite pleasantly, the first three terms coincide with terms from previous steps.  Since we have already established that they must be integers, so is the last one; denote it by  $m$. Thus we have three integers defined by the geometry of the construction:
\begin{equation}    \label{eq:2.4}
  n = \dfrac{4+h^2B^2}{4\rho}, \quad k = \rho B^2,
	\quad m = hB^2.
\end{equation}
Integers $k, n$ and $m$ are not independent;  take the definition for $n$ and eliminate $h$ and $r$ from it to get
\[
  4nk = 4B^2 + m^2,
\]
from which it follows immediately that $m$ must be even, say $m = 2\mu$.  Reduce the common factor of 4 to get the ``master equation'' \eqref{eq:2.1}.

As to the \textbf{constraints}, the order of the curvatures gives three inequalities:
\begin{equation}    \label{eq:2.5}
  B_1 \ge B \Rightarrow k \ge 0, \quad
	B_2 \ge B_1 \Rightarrow n \ge 0, \quad
	B_3 \ge B_2 \Rightarrow k \ge 2\mu.
\end{equation}
The additional upper bound for $\mu$ comes from the fact that $k$ takes its greatest value at $k = \sqrt{B^2+\mu^2}$
\,.  Thus the last inequality of \eqref{eq:2.5},  $k \geq  2\mu$,  implies:
\[
  \sqrt {B^2+\mu^2} \ge 2\mu
\]
and therefore (after squaring)
\[
  B^2 > 3\mu^2.
\]
This ends the proof.  
\end{proof}

\section{Symbols of the circles in an Apollonian gasket}
\label{s:3}

The \textbf{symbol of a circle} \cite{JK1,JK2} is a formal fraction  
\[
  \dfrac{\dot x, \dot y}{b}
\]
where $b = 1/r$ denotes the bend (signed curvature) of the circle and the position of the center is 
\[
  (x,y) = \left( \dfrac{\dot x}{b} ,\, 
	\dfrac{\dot y}{b} \right).
\]
By \textbf{reduced coordinates} we mean the pair $(\dot x, \dot y)$.  In the case of the Apollonian Window (packing with the major curvatures $(-1,\ 2,\ 2,\ 3,\ 3))$, the reduced coordinates and the bend of each circle are integers, see Figure \ref{fig:fig-5}.

\begin{figure}[htbp] 
  \centering
  \includegraphics[width=5.67in,keepaspectratio]{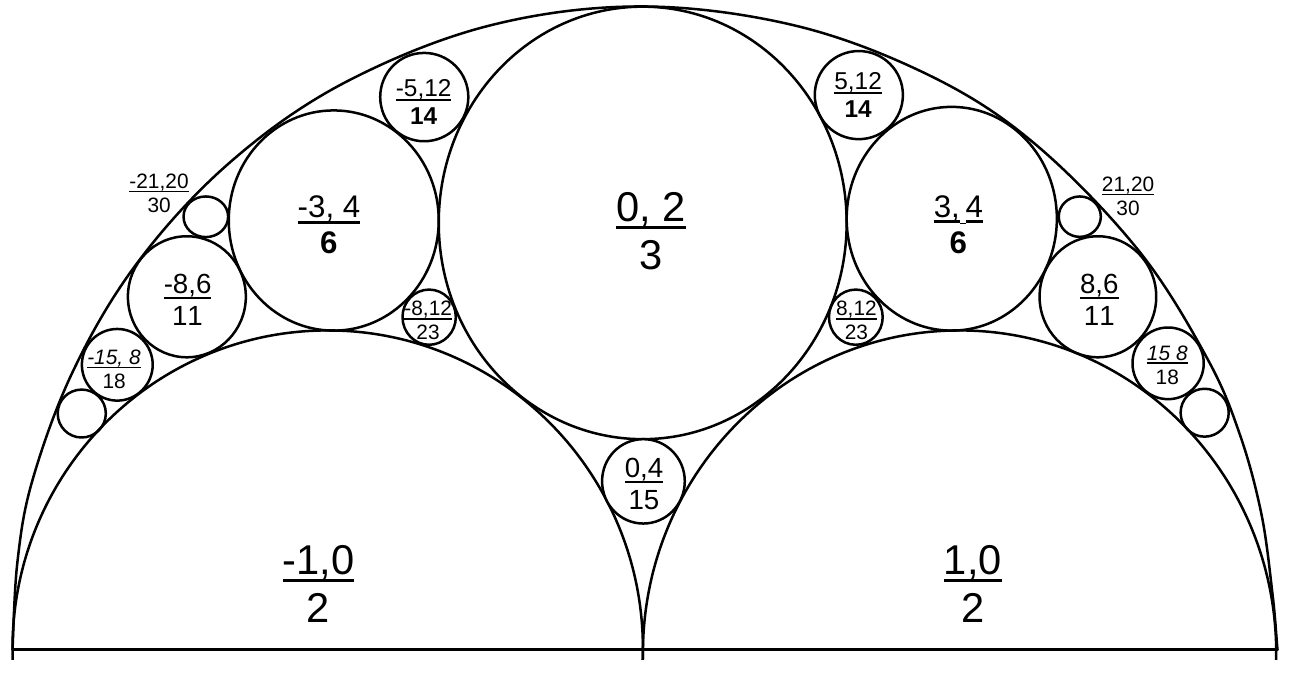}
  \caption{Apollonian Window with labels.}
  \label{fig:fig-5}
\end{figure}

Boyd's linearization \eqref{eq:1.2} for bends in Descartes configurations holds for the reduced coordinates, since they satisfy a quadratic equation  due to the extended Descartes Circle Theorem \cite{GLMWY1, GLMWY2, GLMWY3}.  Thus:
\begin{equation}    \label{eq:3.1}
  \dot x_4 + \dot x'_4 = 2(\dot x_1 + \dot x_2
	+ \dot x_3) \hbox{ \ and \ }
  \dot y_4 + \dot y'_4 = 2(\dot y_1 + \dot y_2
	+ \dot y_3).
\end{equation}

\begin{remark}
The geometric interpretation of the integrality of the labels.  Every circle in the Apollonian Window is an integer multiple of its radius above the horizontal axis and from the vertical axis.  This is a generalization of Pappus' Arbelos Theorem \cite{Wil} for a single chain in an arbelos-like figure, illustrated in Figure \ref{fig:fig-6}.
\end{remark}

\begin{figure}[htbp] 
  \centering
  \includegraphics[width=2.3in,keepaspectratio]{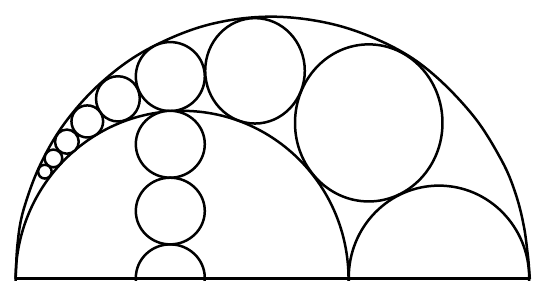}
  \caption{Classical Pappus' Arbelos Theorem: the base is divided randomly, yet every circle in the chain is an even multiple of its radius above the axis.}
  \label{fig:fig-6}
\end{figure}

The question is whether the same may be expected for other integer Apollonian disk packings, that is: Are the three numbers in the label all integers?
\\

\newpage

\begin{proposition}   \label{prp:3.1}  
In the case of the coordinate system with the center located at the center of circle $B$ (see Figure \ref{fig:fig-4}), the labels for the Apollonian gasket generated from $(B, n, k, \mu)$ are as presented in Figure \ref{fig:fig-7}.
\end{proposition}

\begin{figure}[tbp] %
   \centering
\begin{picture}(275,175)
\put(0,0){\includegraphics[width=3in,keepaspectratio]{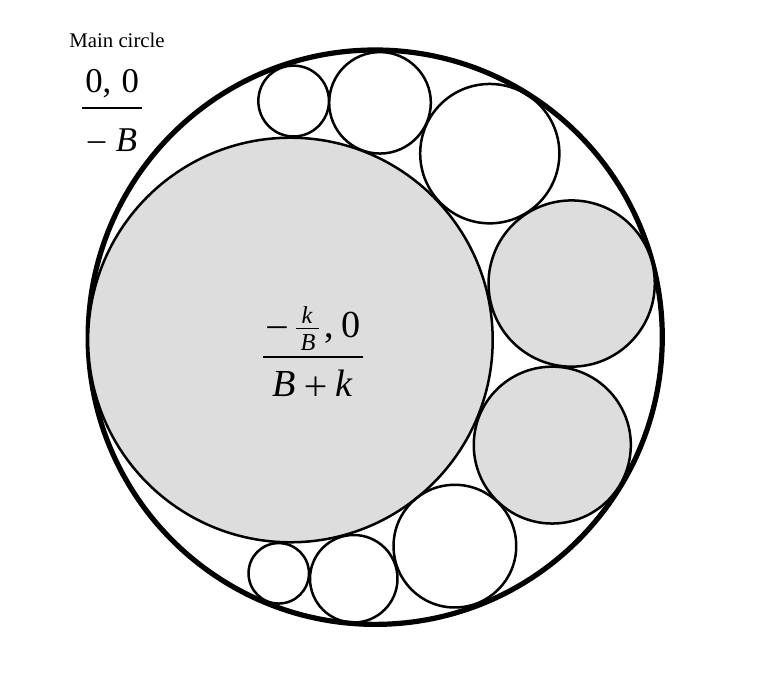} }   
\put(170,155){$\dfrac{\frac{B^2-(k+\mu)^2}{Bk}, 
	\frac{2(k+\mu)}{k}}{B+k+n+2\mu}$}
\put(190,105){$\dfrac{\frac{B^2-\mu^2}{Bk}, \frac{2\mu}{k}}
	{B+n}$}
\put(175,40){$\dfrac{\frac{B^2-(k-\mu)^2}{Bk}, 
	-\frac{2(k-\mu)}{k}}{B+k+n-2\mu}$}
\end{picture}
  \caption{Labels of circles in the Apollonian packing $[B,\,\mu,\, n,\, k]$}
  \label{fig:fig-7}
\end{figure}
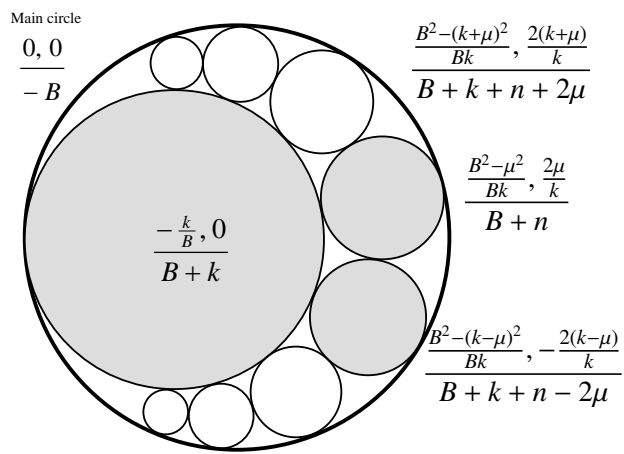

\begin{proof}
Direct calculations.
\end{proof}

One can check that the circles of curvatures $B + k + n \pm\mu$ are Boyd-dual with respect to the triple of mutually tangent circles of bends $(-B,\ B+n,\ B+k)$.

The above proposition implies that the reduced coordinates for circles are fractional unless $B|k$  (see the symbol for $B_1$  in the figure). Since we have also $2\mu \leq k$ (constraints) and the need for $k|2\mu$ (see the symbols for the reduced $y$-coordinate), we conclude that we would need $\mu = 0$ or $k = 2\mu$. 

This happens in two cases: Apollonian Window $(-1,2,2,3,3)$, and Apollonian Strip $(0,0,1,1,1)$, that is for $[B,\mu,k,n]$ equal $[1,0,1,1]$ or $[0,0,0,1]$, respectively.  However, if $k|\mu$  (which happens ``often'', whenever $\mu=0$), all reduced $y$-coordinates are integer.  This is a special (integral) case of Pappus' observation, since $\mu=0$ corresponds to the pattern shown in Figure \ref{fig:fig-6}, see Appendix). 
\\

\newpage

\section{Pythagorean triples in Apollonian gaskets}
\label{s:4}

Given two tangent circles $C_1$ and $C_2$, we construct a triangle whose hypotenuse joins the centers and the other sides of which are horizontal or vertical with respect to some fixed axes.
\[
\includegraphics[width=1.5in,keepaspectratio]{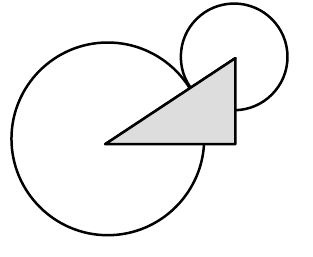}
\]
We shall associate with this triangle a triple of numbers, namely
\begin{equation}    \label{eq:4.1}
  (\Delta,\Gamma,H) = (b_1 \dot x_2 - b_2 \dot x_1,\,
	b_1 \dot y_2 - b_2 \dot y_1, \, b_1 + b_2).
\end{equation}
The actual size of the triangle's sides is the above triple scaled down by the factor of $b_1b_2$.  Hence the \textbf{symbol} for such triangles:
\begin{equation}    \label{eq:4.2}
  \dfrac{\Delta,\Gamma,H} {b_1b_2}\,.
\end{equation}
Clearly, $\Delta^2 + \Gamma^2 = H^2$.  If the reduced coordinates are integers, then the triples \eqref{eq:4.1} are Pythagorean triples!  The Apollonian Window thus contains infinitely many Pythagorean triples \cite{JK1,JK2}. A few are displayed in Figure \ref{fig:fig-8b}.

\begin{figure}[h!] %
  \centering
   \includegraphics[width=4.5in,keepaspectratio]{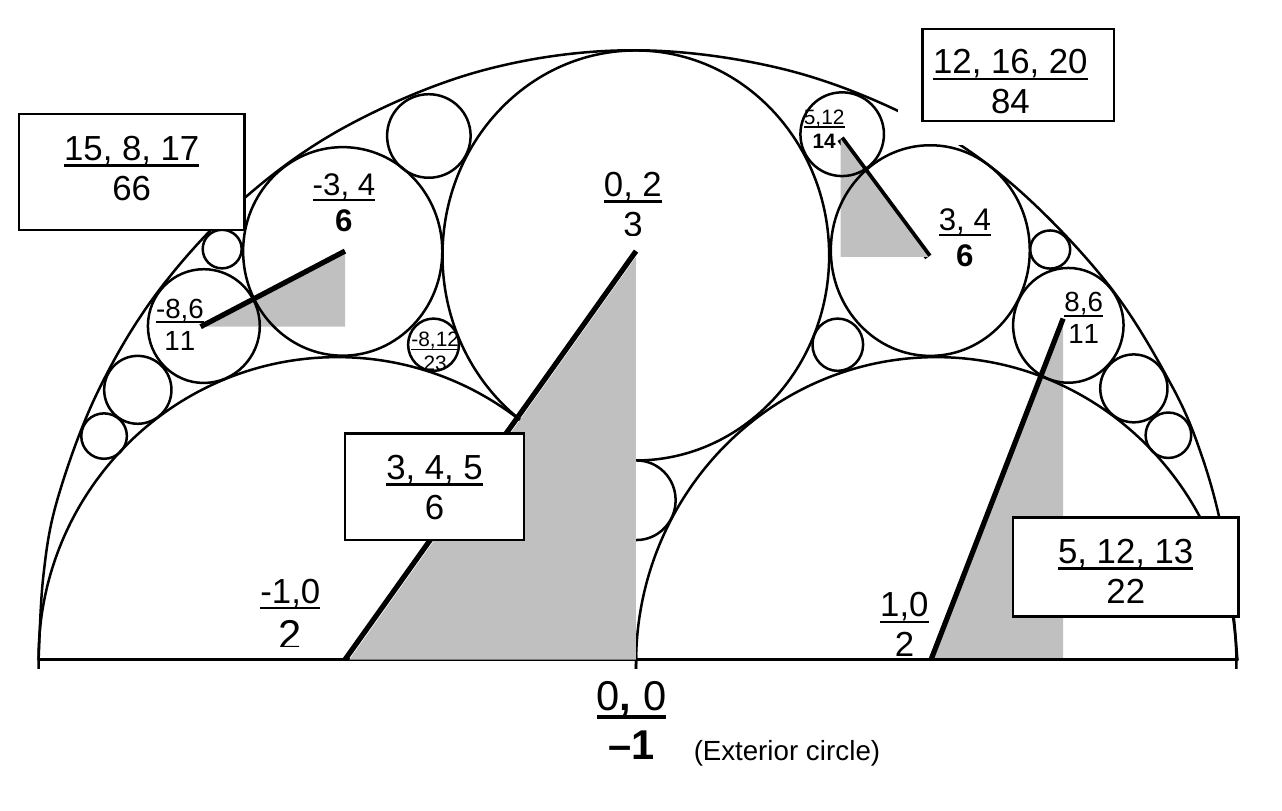} 
  \caption{Pythagorean triples in the Apollonian Window}
  \label{fig:fig-8b}
\end{figure}

Now, for the arbitrary integer packing.  Consider the four major circles $B_0$, $B_1$, $B_2$, $B_3$ in an Apollonian gasket (Figure \ref{fig:fig-6}).  Pairwise, they determine six right triangles. Each of the thick segments in Figure \ref{fig:fig-9} represents the hypotenuse of one of them.  This sextet will be called the \textbf{principal frame} for the gasket.

\begin{figure}[h!] %
  \centering
\begin{picture}(275,225)
\put(-50,0){\includegraphics[width=3.5in,keepaspectratio]{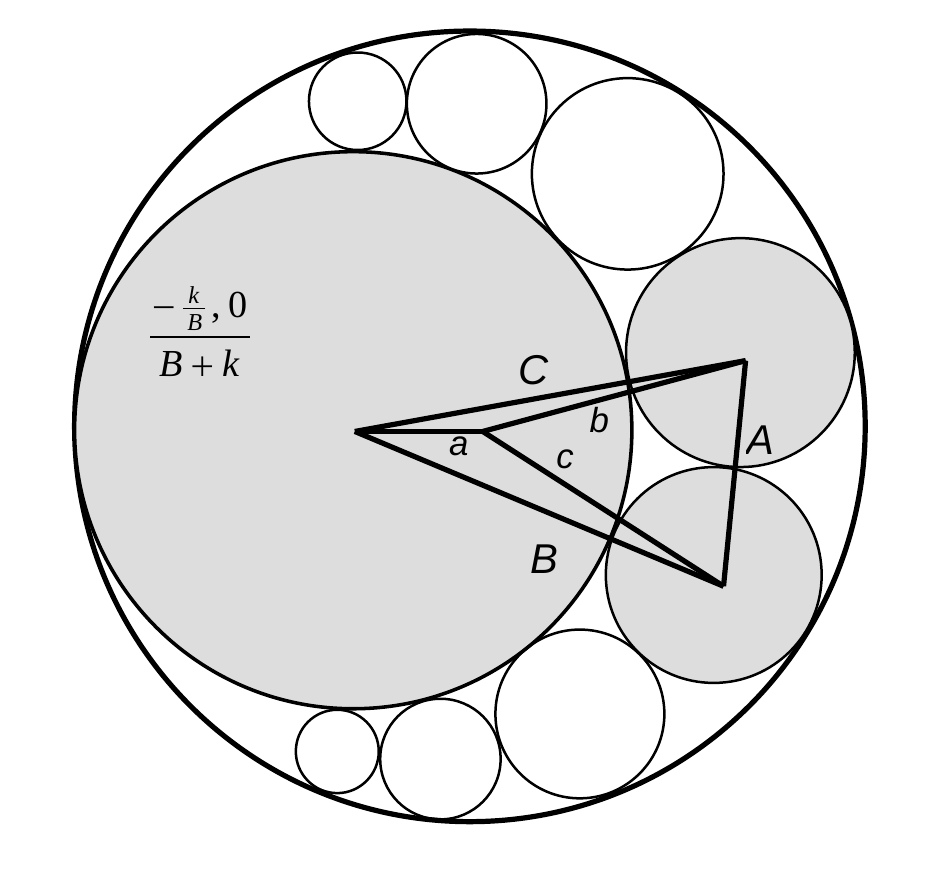}  }
\put(183,150){$\dfrac{\frac{B^2-\mu^2}{Bk}, \frac{2\mu}{k}}
	{B+n}$}
\put(168,35){$\dfrac{\frac{B^2-(k-\mu)^2}{Bk}, 
	-\frac{2(k-\mu)}{k}}{B+n+k-2\mu}$}
\end{picture}
  \caption{Labels of circles in the Apollonian packing $[B,\,\mu,\,n,\,k]$}
  \label{fig:fig-9}
\end{figure}

We shall prove later that if the corresponding triples $\Delta$, $\Gamma$, $H$ are integers, so are all triples in the gasket, But first:

\newpage

\begin{proposition}	\label{prp:4.1}
The labels of the six right triangles in an Apollonian gasket whose hypotenuses join the pairs of the first four greatest circles are, in notation as in Figure \ref{fig:fig-7} and \ref{fig:fig-9}, as follows:
\begin{equation}    \label{eq:4.3}
\begin{aligned}
  a &= \dfrac{k,\,0,\,k}{-B(B+k)} \\
  b &= \dfrac{n-2B^2/k, \ -2\mu B/k, \ n}{-B(B+n)} \\
  c &= \dfrac{k+n-2\mu-2B^2/k, \ 
	2B-2\mu B/k, \ k+n-2\mu}{-B(B+k+n-2\mu)} \\
  A &= \dfrac{-2b-k+2\mu+4B\mu/k, \ -2(B-n+\mu+2B^2/k), \ 
	2B+2n+k-2\mu}{(B+n)(B+k+n-2\mu)}	\\
  B &= \dfrac{-2B+n-2\mu-2B^2/k, \ 2(B+k-\mu-2B\mu/k), \ 
	2B+n+2k-2\mu}{-B(B+k+n-2\mu)}	\\
  C &= \dfrac{2B+k-n+2B^2/k, \ 2\mu+2\mu B/k, \ 2B+k+n}{(B+k)(B+n)}
\end{aligned}
\end{equation}
\end{proposition}

\begin{proof}
Direct calculation from the data of Proposition \ref{prp:3.1}.
\end{proof}

The master equation is used to bring the triples to a form where the only fractional terms contain B/k as a factor.  Note that, in general, each of the entries for the reduced coordinates is a linear combination with integer coefficients of $B, k, n$ and $\mu$ plus the only, possibly, non-integer term, a multiple of $2B\mu /k$ or $2B^2/k$.  The fact that it can be done proves the following theorem:

\begin{theorem}		\label{th:4.2}
If  $k|2B^2$ (or, equivalently $k|2B\mu$)  then all triples (4.1) in the integer Apollonian gasket (B, $\mu$, k) are integers (form Pythagorean triples). 
\end{theorem}

\begin{proof}
The following implication is direct
\[
  k|2B^2 \Rightarrow k|(2\mu k-2\mu^2) \Rightarrow k|2\mu^2.
\]
Combining the premise with the result:  $k^2|2B^22\mu^2 \Rightarrow  k^2|4(B\mu)^2 \Rightarrow k|2\mu B$.  Every entry of \eqref{eq:4.3} may be expressed as a linear integral combination of $B$, $k$, $n$, $\mu$ and an extra term of either $2B\mu/k$ or $2B\mu/k$.  This proves the claim.  
\end{proof}

Integer Pythagorean triples happen frequently. For instance, each of the conditions: $\mu = 0$, or $k = 1$, or $k = 2$, is sufficient.

\section{Linear recurrence for Pythagorean triples in an Apollonian gasket}
\label{s:5}

Consider four circles in Descartes configuration, $C_1,\ldots ,C_4$  (see Figure \ref{fig:fig-10}).  The four centers determine six segments that we view as hypotenuses of triangles. They give, after rescaling, six Pythagorean triangles with sides denoted as follows:
\begin{equation}    \label{eq:5.1}
\begin{array}{rlll}
  \hbox{horizontal side:} &\quad 
                      \Delta_{ij}  &=  \ b_ib_j(x_i-x_j) 
	                               &= \ \dot x_ib_j - \dot x_jb_i \\
  \hbox{vertical side:} &\quad
                    \Gamma_{ij} &= \ b_ib_j(y_i-y_j) 
                                         & = \ \dot y_ib_j - \dot y_jb_i \\
  \hbox{hypotenuse:} &\quad
                               B_{ij} &= \ b_ib_j(1/b_i+1/b_j) 
                                         &= \ b_i+b_j
\end{array}
\end{equation}
with $i,j = 1,\ldots 4$.  Note that $\Delta_{ij} = -\Delta_{ji}$ and $\Gamma_{ij} = -\Gamma_{ji}$  but $B_{ij} = B_{ji}$.  They form the frame of the Descartes configuration.  

Now, complement the picture with a circle $C'_4$, Boyd-dual to $C_1$ (see Figure \ref{fig:fig-10}).  This leads to a new frame, the frame of the Descartes configuration $C'_4$, $C_2$, $C_3$, $C_4$. Quite interestingly, the elements of the new frame are linear combinations of the elements of the initial frame.

\begin{theorem}		\label{th:5.1}
Following the notation of \eqref{eq:5.1} and Figure \ref{fig:fig-7}, the transition of frames is given by
\[
\begin{aligned}
  \Delta'_{41} &= -\Delta_{41} + 2\Delta_{21} - 2\Delta_{13} \\
  \Gamma'_{41} &= -\Gamma_{41} + 2\Gamma'_{21} - 2\Gamma'_{13} \\
  B'_{41} &= -B_{41} + 2B_{12} + 2B_{13}
\end{aligned}
\]
\end{theorem}

\begin{proof}
Using the linear relations \eqref{eq:1.2} and \eqref{eq:3.1}, we can express the position and the curvature of the new, fifth, circle in terms of the initial four:  
\[
\begin{aligned}
  \dot x'_4 &= 2\dot x_1 + 2\dot x_2 + 2\dot x_3 - \dot x_4 \\
  \dot y'_4 &= 2\dot y_1 + 2\dot y_2 + 2\dot y_3 - \dot y_4 \\
  b'_4 &= 2b_1 + 2b_2 + 2b_3 - b_4
\end{aligned}
\]
We can calculate the Pythagorean vectors for each pair $C'_4C_i$,  $i=1,2,3$. By some magic, due to the following regroupings and adding zeros, these can be expressed in terms of the initial four Pythagorean vectors only, as shown here for the pair $C'_4C_1$:
\[
\begin{aligned}
  \Delta'_{41} &= \dot x'_4 b_1 - \dot x_1b'_4 \\
	&= (2\dot x_1 + 2\dot x_2 + 2\dot x_3 - \dot x_4)
		b_1 - \dot x_1(2b_1 + 2b_2 + 2b_3 - b_4) \\
	&= \dot x_1b_4 - \dot x_4b_1 + 2(\dot x_2b_1 
		- \dot x_1b2_1) + 2(\dot x_3b_1 - \dot x_1b_3) \\
	&= -Delta_{41} + 2\Delta_{21} - 2\Delta_{13} \, .
\end{aligned}
\]
The $\Gamma$'s follow the same argument.  Only $B'$s are slightly different: 
\[
\begin{aligned}
  B'_{41} &= b'_4 + b_1 \\
	&= -b_4 + 3b_1 + 2b_2 + 2b_3  \\
	&= 4b_1 + 2b_2 + 2b_3 - (b_1+b_4)  \\
	&= -B_{41} + 2B_{12} + 2B_{13} \, .
\end{aligned}
\]
Thus the transformation is linear in the entries of $\Delta$, $\Gamma$, and, $B$, with integer coefficients, as stated.  
\end{proof}

\begin{figure}[h!] %
  \centering
  \includegraphics[width=3.5in,keepaspectratio]{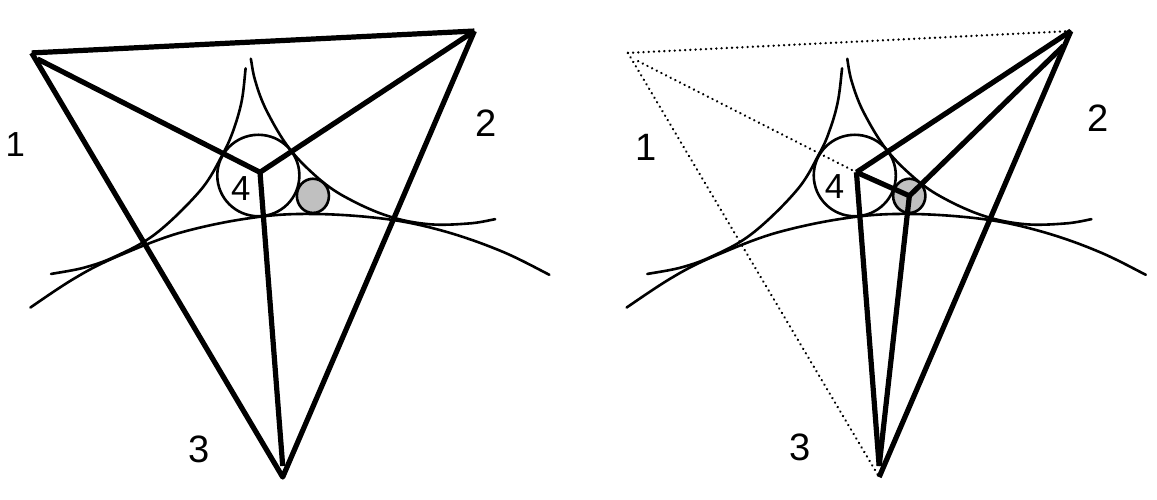}
  \caption{(a) A frame (b) and an adjacent frame.  Circle 4$'$ is shaded.}
  \label{fig:fig-10}
\end{figure}

Actually, we have proved a little more:

\begin{corollary}	\label{cor:5.2}
If the (scaled) triangles formed by the initial Descartes configuration are integer, so are all of the triangles of the Apollonian window.
\end{corollary}

\textbf{Matrix description:}  Now we may rephrase our findings on transitions between frames in terms of matrices.  The data for the Pythagorean triangles in each frame may be expressed as three vectors (columns), e.g., for the initial frame: $\Delta = [\Delta_{41}, \Delta_{42}, \Delta_{43}, \Delta_{12}, \Delta_{23}, \Delta_{31}]^{T}$ and similarly for $\Gamma$ and  $B$. Similarly for the second frame we have $\Delta'$,  $\Gamma'$ and  $B'$.

Here are the three matrices of the transition from a frame to the corresponding subsequent frame:
\[
  A = \begin{bmatrix}
	\hbox{\small\begin{tabular}{rrr|rrr}
	1 &0 &0 &0 &0 &0 \\
	0 &0 &0 &0 &0 &1 \\
	0 &0 &0 &0 &$-$1 &0 \\ \hline
	0 &0 &0 &1 &$-$2 &$-$2 \\
	2 &0 &1 &0 &2 &0 \\
	$-$2 &$-$1 &0 &0 &0 &2 
		\end{tabular} }
	\end{bmatrix}
\quad
  B = \begin{bmatrix}
	\hbox{\small\begin{tabular}{rrr|rrr}
	0 &0 &0 &0 &0 &$-$1 \\
	0 &1 &0 &0 &0 &0 \\
	0 &0 &0 &1 &0 &0 \\ \hline
	0 &$-$2 &$-$1 &2 &0 &0 \\
	0 &0 &0 &$-$2 &1 &$-$2 \\
	1 &2 &0 &0 &0 &2 
		\end{tabular} }
	\end{bmatrix}
\]
\[
  C = \begin{bmatrix}
	\hbox{\small\begin{tabular}{rrr|rrr}
	0 &0 &0 &0 &1 &0 \\
	0 &0 &0 &$-$1 &0 &0 \\
	0 &0 &1 &0 &0 &0 \\ \hline
	0 &1 &2 &2 &0 &0 \\
	$-$1 &0 &$-$2 &0 &2 &0 \\
	0 &0 &0 &$-$2 &2 &1 
		\end{tabular} }
	\end{bmatrix}
\]
(vertical and horizontal lines in the matrices are drawn for easier inspection).   Now acting with them on the initial vectors $\Delta$ and $\Gamma$ one may reconstruct all the interior triangles.  Vectors  $\Delta$ and $\Gamma$ may be combined into a single matrix, $T = [\Delta|\Gamma]$.   
\[
  T = \begin{bmatrix}
	\Delta_{41}  &\Gamma_{41}  \\
	\Delta_{42}  &\Gamma_{42}  \\
	\Delta_{43}  &\Gamma_{43}  \\
	\Delta_{12}  &\Gamma_{12}  \\
	\Delta_{23}  &\Gamma_{23}  \\
	\Delta_{31}  &\Gamma_{31}  \\
	\end{bmatrix}
\quad
  B' = \begin{bmatrix}
	B_{41}  \\
	B_{42}  \\
	B_{43}  \\
	B_{12}  \\
	B_{23}  \\
	B_{31}
	\end{bmatrix}
\]
The vector $B$ is transformed to $B'$ by matrices like $A, B, C$, where all of the entries are to be  replaced by their absolute values.

We conclude with two examples of Apollonian packings in which every triangle for adjacent circles determines a Pythagorean triple, see Figures \ref{fig:fig-10} and \ref{fig:fig-11}.  Only some of the triangles are displayed.

\section{Concluding remarks}
\label{s:6}

Integral Apollonian disk packings have a topic of much interest for a while \cite{GLMWY1}-\cite{GLMWY3}, \cite{Sar1}--\cite{Sar2} and their occurrence has been analyzed. Yet it is not clear whether the integral Apollonian disk packings admit coordinate systems in which the reduced coordinates are integral, except the known cases of Apollonian Strip and Apollonian Window.  We have however a frequent occurrence of gaskets that admit integral Pythagorean triples, constructed in Sections \ref{s:4} and \ref{s:5}. 

Yet another intriguing remark concerns the master equation 
\[
  B^2 + \mu^2 - nk = 0.
\]
If we see it as $B^2 + \mu^2 = \hbox{hypotenuse}^2$  --  one may identify the corresponding (scaled) triangle in Figure \ref{fig:fig-4} as the one with legs being the diameter $2/B$ of the encompassing circle and altitude $h$ of the center of circle $B_2$ -- the inversive image of circle $B_2$.  Each solution to the master equation may also be viewed as an integral isotropic vector $[B, \mu, n, k]^T$ in the Minkowski space $\mathbb R^{3,1}$  \textit{a la} Pedoe map from circles to vectors, see \cite{JK2}.  Such vectors represent points in this Minkowski space -- degenerated circles of radius 0.

\begin{figure}[tbp] %
  \centering
  \includegraphics[width=5.67in,keepaspectratio]{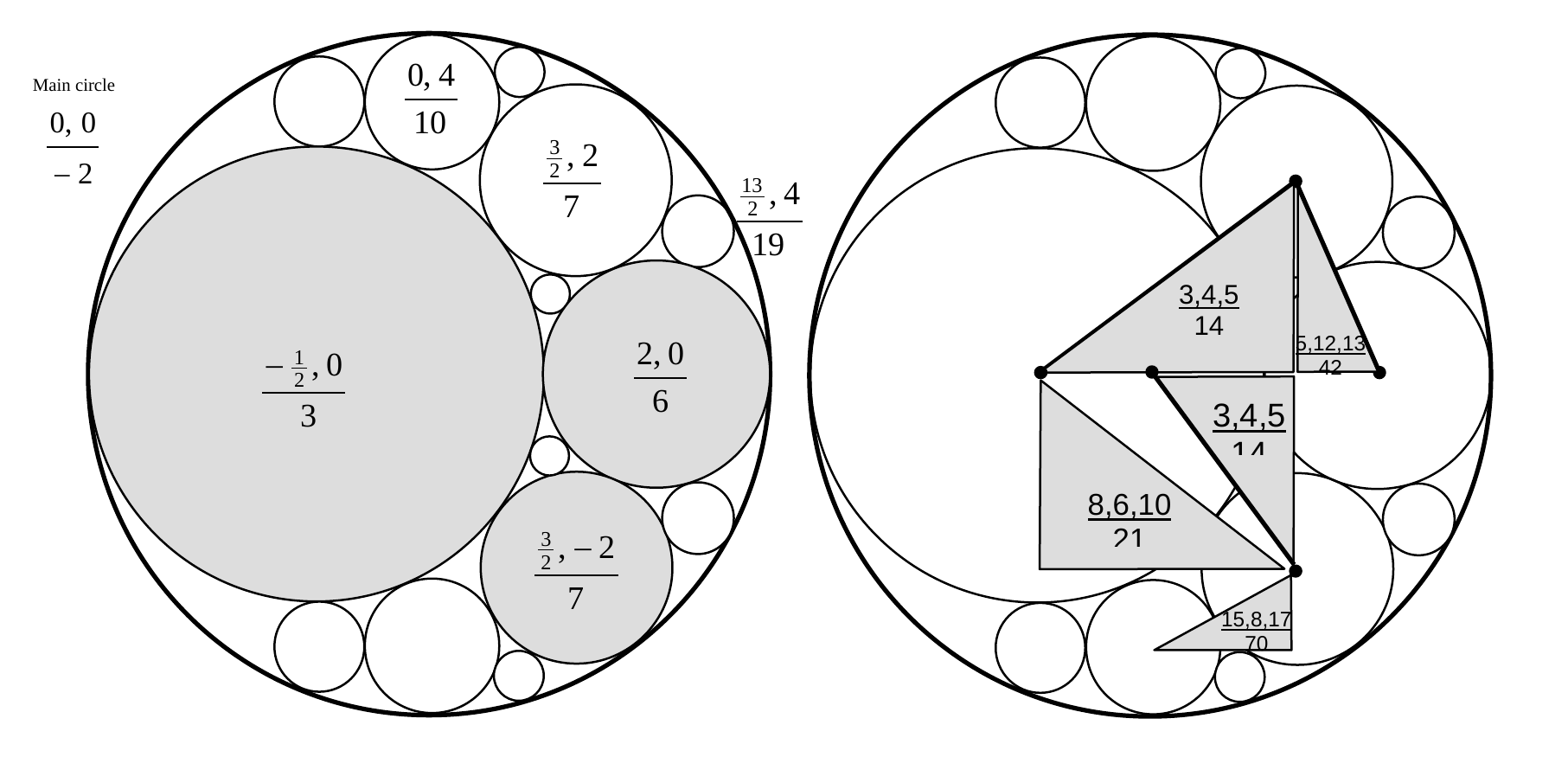}
  \caption{Arrangement $(b_0,b_1,b_2,b_3) = (-2,3,6,7)$ or $(B,\mu,k,n) = (2,4,4,5)$.  Left side shows symbols of the principal and other circles, right side shows Pythagorean triples.}
  \label{fig:fig-11}
\end{figure}

\begin{figure}[tbp] %
  \centering
  \includegraphics[width=5.67in,keepaspectratio]{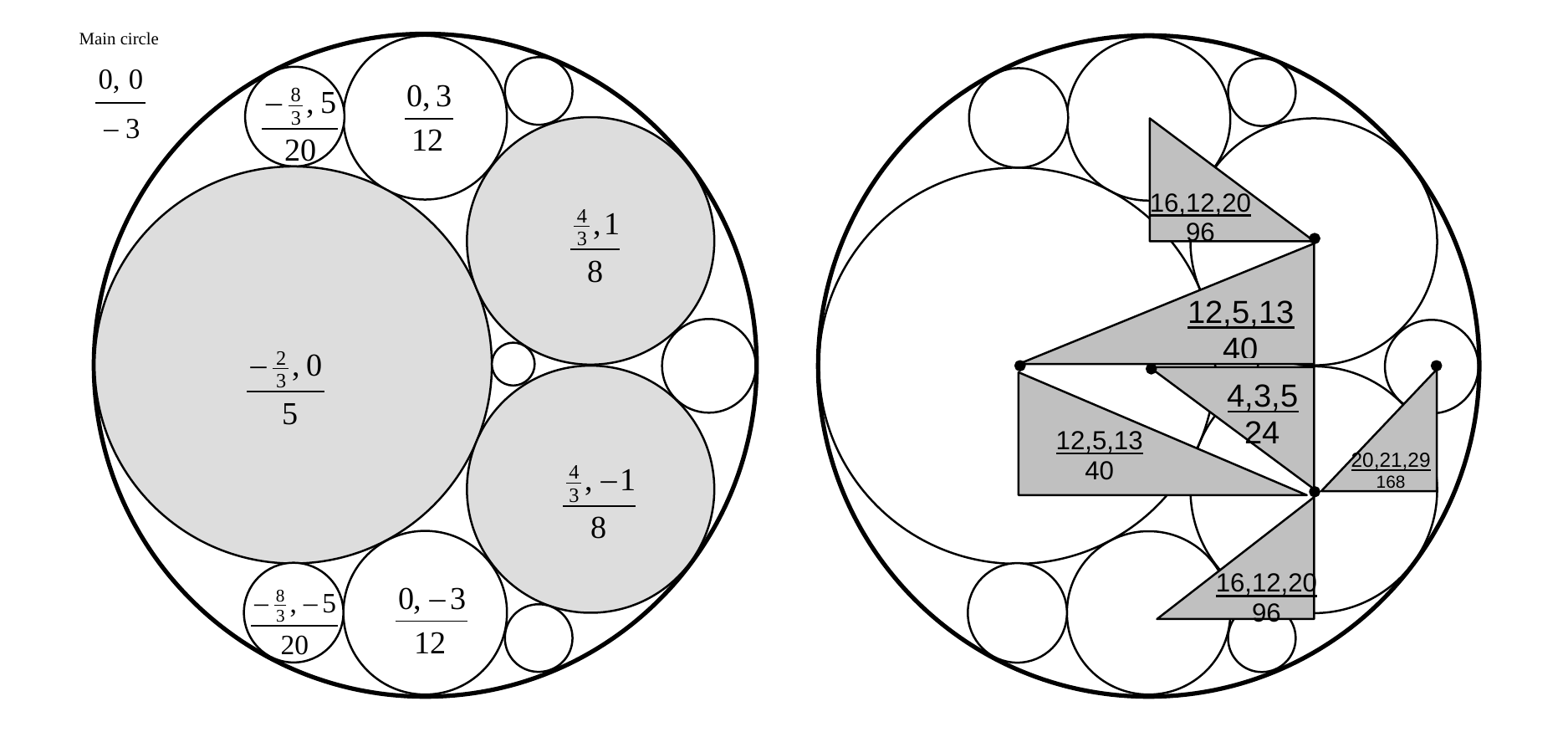}
  \caption{Arrangement $(b_0,b_1,b_2,b_3) = (-3,5,8,8)$ or $(B,\mu,k,n) = (3,4,5,5)$.  Both symbols (left) and  Pythagorean triples (right).  The triangles are reported with absolute values of sides.}
  \label{fig:fig-12}
\end{figure}

\clearpage

\renewcommand{\thefigure}{A\arabic{figure}}
\setcounter{figure}{0}
\renewcommand{\thetable}{A\arabic{table}}
\setcounter{table}{0}

\appendix
\section{List of integral packings}
\label{s:apndx}

Here we provide a list of the first 200 integer Apollonian packings.  First, however, a few words on symmetry of these arrangements.  

\begin{figure}[h!] %
  \centering
  \includegraphics[width=4.5in,keepaspectratio]{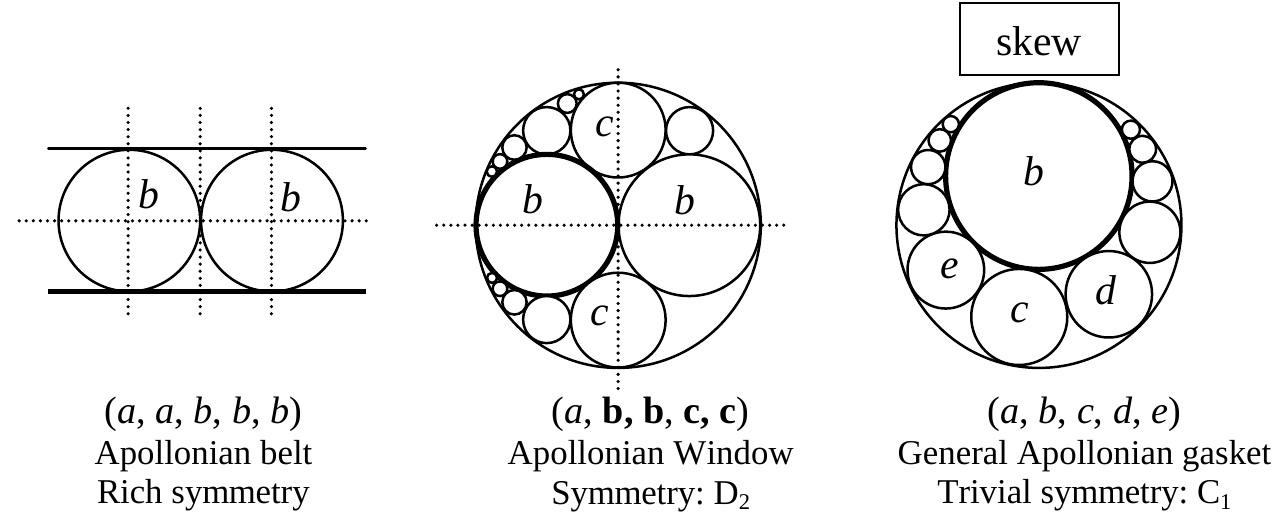}
  \includegraphics[width=4.5in,keepaspectratio]{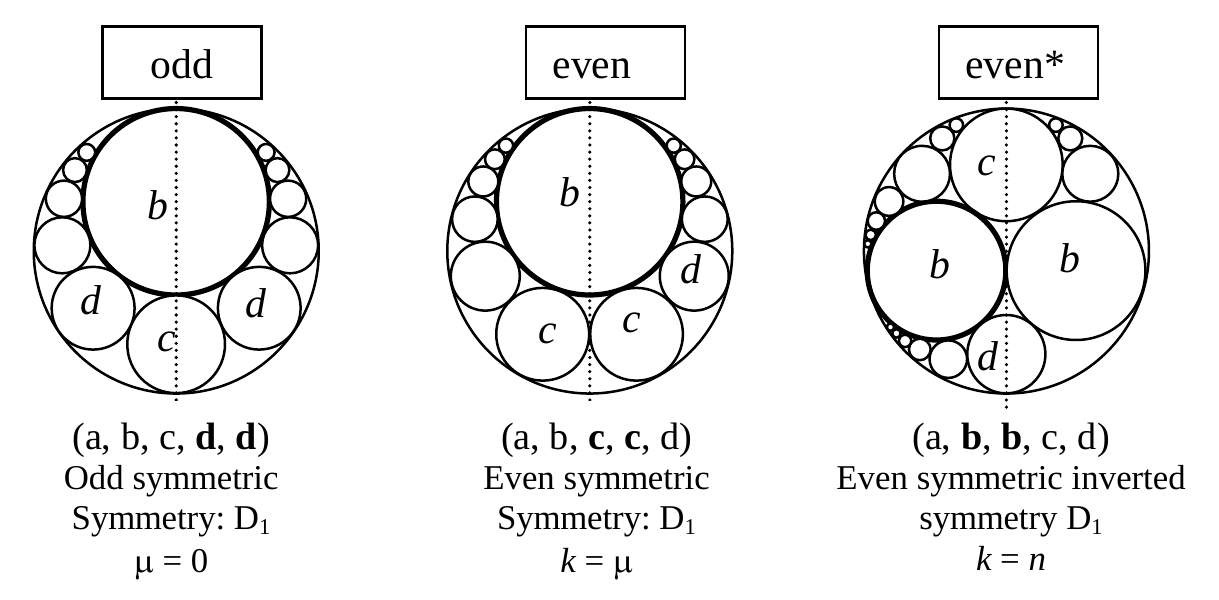}
  \caption{Types of symmetries and arrangements.
	The second circle, $b$, is thicker.}
  \label{fig:fig-14a}
\end{figure}

Only two integral Apollonian packings have rich symmetry:  that of Apollonian Strip, and of Apollonian Window $(D_2)$.  Otherwise we have two cases, that of symmetry $C_1$ and $D_1$, which will be labeled as follows: 

\begin{enumerate}
\item[]\textbf{Skew} means no nontrivial symmetry (thus symmetry of type $C_1$). 
\end{enumerate}

 For cases with single axis symmetry ($D_1$) we shall distinguish two cases: 
\begin{enumerate}
\item[]\textbf{Odd} \ \ -- mirror symmetry $D_1$, with exactly three circles cut by the axis of symmetry

\item[]\textbf{Even} \ -- mirror symmetry with infinitely many circles on the axis

\item[]\textbf{Even*} -- even symmetry that is obtained ``accidentally'' by algorithm in Figure~\ref{fig:fig-3a}, where the two greatest congruent circles are not images of the congruent disks in the Apollonian strip
\end{enumerate}

Consult Figure A1 for visualization. 

Yet another attribute of an Apollonian packing is a \textbf{shift}, which we define as the degree the disk on the horizontal axis in the Apollonian strip is raised above this axis in terms of the fraction of its radius.  
\[
  \hbox{shift} = h/\rho = \dfrac{2\mu}{k}\,.
\]
See Figure \ref{fig:fig-4}.  It has values in the interval [0,1] and it measures the degree the gasket is off the axial symmetry:  $0$ for \textit{odd} symmetry $D_1$ and $1$ for \textit{even} symmetry $D_1$. The fractional value indicates lack of axial symmetry, except the cases of ``accidental'' even* symmetry, when $n = k$ (this gives $B_1 = B_2$).

Table A1 displays first 24 integral Apollonian gaskets including symmetry type, principal disk quintet, the parameters $B$, $k$, $n$, $\mu$, and the shift factor. The first even* case is also included.

Table A2 contains a longer list of integer Apollonian gaskets for the principal curvatures from 1 through 32.  Curvature quintets (in brackets) are followed by shift factor.

\begin{table}[htbp] %
  \centering
  \includegraphics[width=5.67in,keepaspectratio]{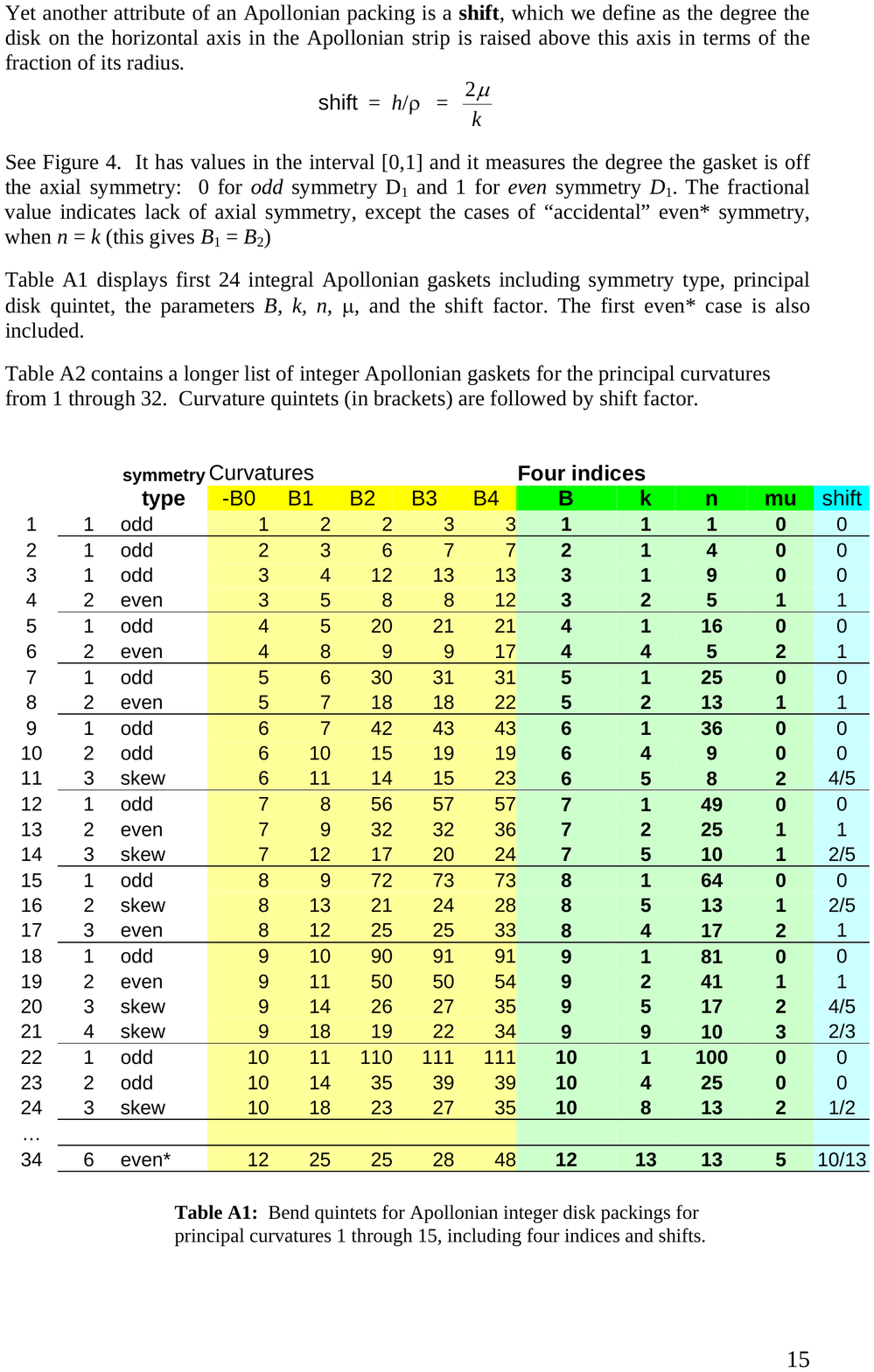}
  \caption{Bend quintets for Apollonian integer disk packings for principal curvatures 1 through 15, including four indices and shifts.}
  \label{fig:tbl-A1}
\end{table}

\begin{table}[htbp] 
  \centering
  \includegraphics[width=5.67in,keepaspectratio]{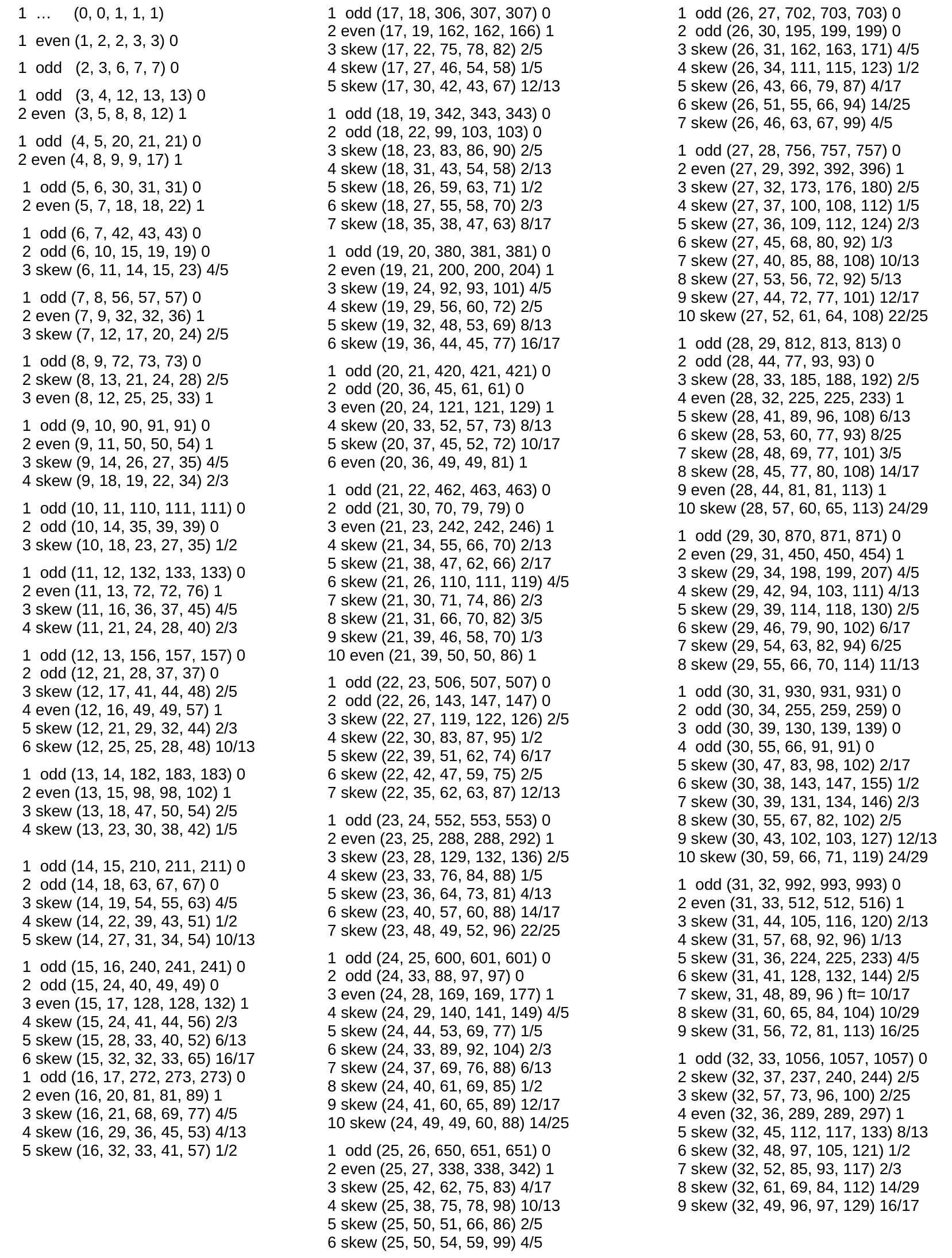}
  \caption{Bend quintets for Apollonian integer disk packings for principal curvatures 1 through 32.}
  \label{fig:tbl-A2}
\end{table}


\begin{thebibliography}{99}
\addcontentsline{toc}{chapter}{References}
\frenchspacing

\bibitem{B}  
D. W. Boyd,
The osculatory packing of a three-dimensional sphere,
\textit{Canadian J. Math.},
\textbf{25} (1973), 303--322.

\bibitem{GLMWY1}  
R. L. Graham, Jeffrey C. Lagarias, C. L. Mallows, A. R. Wilks, and C. H. Yan,
Apollonian circle packings: Geometry and group theory I. The Apollonian group, \textit{Discrete \& Computational Geometry},
\textbf{34}(4): (2005), 547--585.

\bibitem{GLMWY2}  
R. L. Graham, Jeffrey C. Lagarias, C. L. Mallows, A. R. Wilks, and C. H. Yan,
Apollonian circle packings: Geometry and group theory II. Super-Apollonian group and integral packings, 
\textit{Discrete \& Computational Geometry},
\textbf{35}(1): (2006), 1--36.

\bibitem{GLMWY3}  
R. L. Graham, Jeffrey C. Lagarias, C. L. Mallows, A. R. Wilks, and C. H. Yan,
Apollonian circle packings: Geometry and group theory III. Higher dimensions, 
\textit{Discrete \& Computational Geometry},
\textbf{35}(1): (2006), 37--72.

\bibitem{JK1}  
Jerzy Kocik,
Clifford Algebras and Euclid's parametrization of Pythagorean triples, \textit{Advances in Appl. Cliff. Alg.} (Mathematical Structures),
\textbf{17} (2007), pp.\ 71--93.

\bibitem{JK2}  
Jerzy Kocik,
A theorem on circle configurations.
arXiv:0706.0372v2 [math].

\bibitem{LMW}  
Jeffrey C. Lagarias, C. L. Mallows and A. Wilks,
Beyond the Descartes circle theorem, 
\textit{Amer. Math. Monthly} 
\textbf{109} (2002), 338--361.   
[eprint: arXiv math.MG/0101066].

\bibitem{Sar1}  
Peter  Sarnak,
Letter to J. Lagarias about Integral Apollonian packings (preprint 2007), \newline
  [www.math.princeton.edu/sarnak/AppolonianPackings.pdf].

\bibitem{Sar2}  
Peter Sarnak, Integral Apollonian Packings, {\it Am. Math. Monthly}, {\bf 118}, no. 4, (2011),  pp. 291--306. 

\bibitem{TL}  
G. F. T\'oth, Jeffrey C. Lagarias,
Guest editors' foreword,
\textit{Discrete \& Computational Geometry},
\textbf{36}(1): (2006), 1--3.

\bibitem{Wil}  
J. B. Wilker,
Inversive Geometry, in: The Geometric Vein, (C. Davis, B. Gr\"unbaum, F.A. Sherk, Eds.), 
Spring-Verlag: New York 1981, pp. 379--442.


\end{thebibliography}
\end{document}